\documentclass[11pt]{amsart}

\usepackage[utf8]{inputenc}
\usepackage{mathtools,amsfonts,amsthm,mathrsfs,amssymb}
\usepackage{textcomp}
\mathtoolsset{showonlyrefs}
\usepackage{bookmark}

\newtheorem{theorem}{Theorem}
\newtheorem{lemma}[theorem]{Lemma}

\newtheorem{conjecture}[theorem]{Conjecture}

\newcommand{\EE}{\mathbb{E}}

\newcommand{\bF}{\mathbf{F}}
\newcommand{\bI}{\mathbf{I}}
\newcommand{\bJ}{\mathbf{J}}
\newcommand{\cI}{\mathcal{I}}

\newcommand{\lam}{\lambda}
\renewcommand{\epsilon}{\varepsilon}
\newcommand{\eps}{\varepsilon}
\renewcommand{\Pr}{\mathbb{P}}
\newcommand{\RR}{\mathbb{R}}

\title[Occupancy fraction, fractional colouring, and triangles]{Occupancy fraction, fractional colouring, and triangle fraction}

\author[E.\ Davies]{Ewan Davies}
\address{Korteweg-De Vries Institute for Mathematics, University of Amsterdam, Netherlands.}
\email{maths@ewandavies.org}
\thanks{(E.\ Davies) The research leading to these results has received funding from the European Research Council under the European Union's Seventh Framework Programme (FP7/2007--2013) / ERC grant agreement \textnumero{} 339109.}
\author[R.\ de Joannis de Verclos]{R\'{e}mi de Joannis de Verclos}
\address{Department of Mathematics, Radboud University Nijmegen, Netherlands.}
\email{r.deverclos@math.ru.nl}
\author[R. J.\ Kang]{Ross J. Kang}
\address{Department of Mathematics, Radboud University Nijmegen, Netherlands.}
\email{ross.kang@gmail.com}
\thanks{(R.\ de Joannis de Verclos, R.\ J.\ Kang) Supported by a Vidi grant (639.032.614) of the Netherlands Organisation for Scientific Research (NWO)}
\author[F.\ Pirot]{Fran\c{c}ois Pirot}
\address{Department of Mathematics, Radboud University Nijmegen, Netherlands and  LORIA, Universit\'e de Lorraine, Nancy, France.}
\email{francois.pirot@loria.fr}

\subjclass[2010]{Primary 05C35, 05D10; Secondary 05C15}

\keywords{Independent sets, fractional colouring, hard-core model}

\begin{document}
\begin{abstract}
Given $\eps>0$, there exists $f_0$ such that, if $f_0 \le f \le \Delta^2+1$, then for any graph $G$ on $n$ vertices of maximum degree $\Delta$ in which the neighbourhood of every vertex in $G$ spans at most $\Delta^2/f$ edges,
\begin{enumerate}
\item 
an independent set of $G$ drawn uniformly at random has at least $(1/2-\varepsilon)(n/\Delta)\log f$ vertices in expectation, and
\item
the fractional chromatic number of $G$ is at most $(2+\varepsilon)\Delta/\log f$.
\end{enumerate}
These bounds cannot in general be improved by more than a factor $2$ asymptotically.
One may view these as stronger versions of results of Ajtai, Koml\'{o}s and Szemer\'{e}di (1981) and Shearer (1983). The proofs use a tight analysis of the hard-core model.
\end{abstract}

\maketitle

\section{Introduction}\label{sec:intro}

Ajtai, Koml\'{o}s and Szemer\'{e}di gave in 1981 a well-known and influential lower bound on the independence number of triangle-free graphs~\cite{AKS81}. 
They also showed the following statement which allows the graph to contain triangles, with a correspondingly weaker bound as the number of triangles grows.
There is some $C>0$ and some $f_0$ such that, in any graph on $n$ vertices of maximum degree $\Delta$ with at most $\Delta^2n/f$ triangles, where $f_0 <f< \Delta$, there is an independent set of size at least $C(n/\Delta)\log f$.
Somewhat later, Alon, Krivelevich and Sudakov~\cite{AKS99} proved a stronger version of this in terms of an upper bound on the chromatic number.
Recently, using a sophisticated ``stochastic local search'' framework, Achlioptas, Iliopoulos and Sinclair~\cite{AIS18+} tightened the result of~\cite{AKS99}, corresponding to a constant $C$ above of around $1/4$ in general\footnote{They also obtained an asymptotic estimate around $1/2$ for $f$ very close to $\Delta^2+1$.}. 
In fact, shortly after the work in~\cite{AKS81}, using a sharper bootstrapping from the triangle-free case, Shearer~\cite{She83} had improved the above statement on independence number essentially\footnote{The results in~\cite{AKS81,She83} are in terms of given average degree.} as follows.

\begin{theorem}[Shearer~{\cite[Equation (8)]{She83}}]\label{thm:trianglecount}
Given $\eps>0$, there exists $f_0$ such that, if $f_0 \le f \le \Delta^2/\eps^2$, then in any graph on $n$ vertices of maximum degree $\Delta$ with at most $\Delta^2n/f$ triangles, there is an independent set of size at least $(1/2-\eps)(n/\Delta)\log f$.
\end{theorem}
\noindent
The case $f=\Delta^{2-o(1)}$ as $\Delta\to\infty$ includes the triangle-free case and yields the best to date asymptotic lower bound on the off-diagonal Ramsey numbers. 
The asymptotic factor $1/2$ cannot be improved above $1$, due to random regular graphs; see Section~\ref{sec:sharpness} for more details on sharpness.

Our main contribution is to give two stronger forms of Theorem~\ref{thm:trianglecount}, one on occupancy fraction (see Theorem~\ref{thm:occupancy} below), the other on fractional chromatic number, combining for the result promised in the abstract. 

\begin{theorem}\label{thm:main}
Given $\eps>0$, there exists $f_0$ such that, if $f_0 \le f \le \Delta^2+1$, then for any graph $G$ on $n$ vertices of maximum degree $\Delta$ in which the neighbourhood of every vertex in $G$ spans at most $\Delta^2/f$ edges,
\begin{enumerate}
\item\label{itm:occupancy}
an independent set of $G$ drawn uniformly at random has at least $(1/2-\varepsilon)(n/\Delta)\log f$ vertices in expectation, and
\item\label{itm:fractional}
the fractional chromatic number of $G$ is at most $(2+\varepsilon)\Delta/\log f$.
\end{enumerate}
\end{theorem}
\noindent
Note that the assumption on the sparsity of the graph $G$ in Theorem~\ref{thm:main} is expressed differently from the one in Theorem~\ref{thm:trianglecount}; it has to be locally constrained. 
This is necessary for~\ref{itm:fractional} to hold, since a graph with at most $\Delta^2 n/f$ triangles can contain a clique of order $\Delta+1$ in the case $n/f \ge \Delta/6$, which is incompatible with~\ref{itm:fractional}. It is however unclear whether the same applies for~\ref{itm:occupancy}. 
The desire to reach more precise conclusions in terms of occupancy fraction and fractional chromatic number motivates strengthening ones assumptions from given average degree and total triangle count (as in~\cite{AKS81,She83}) to given maximum degree and bounded \emph{local} triangle count (as we do here, and in~\cite{AKS99}). Under these stronger assumptions our Theorem~\ref{thm:main} is easily seen to imply the conclusion of Theorem~\ref{thm:trianglecount}, but in fact by deleting vertices contained in too many triangles either conclusion of Theorem~\ref{thm:main} directly implies Theorem~\ref{thm:trianglecount}; we give this implication explicitly in Section~\ref{sec:sparsenbhds}.

The proof of Theorem~\ref{thm:main} is by an analysis of the hard-core model, a natural probability distribution on independent sets in a graph. In Section~\ref{sec:sharpness}, we give some indication that our application of this analysis is essentially tight. The same method was used for similar results specific to triangle-free graphs~\cite{DJPR18,DJKP20}; to an extent, the present work generalises that earlier work.

Theorem~\ref{thm:main}\ref{itm:fractional} and the results in~\cite{AIS18+} hint at their common strengthening.
\begin{conjecture}\label{conj:chi}
Given $\eps>0$, there exists $f_0$ such that, if $f_0 \le f \le \Delta^2+1$, then any graph of maximum degree $\Delta$ in which the neighbourhood of every vertex spans at most $\Delta^2/f$ edges has (list) chromatic number at most $(2+\varepsilon)\Delta/\log f$.
\end{conjecture}
\noindent
In Section~\ref{sec:basic}, motivated by quantitative Ramsey theory, we briefly discuss a more basic problem setting in terms of bounded triangle fraction.

\subsection{Notation and preliminaries}\label{sub:prelim}

We write $\cI(G)$ for the set of independent sets (including the empty set) of a graph $G$.

Given $\lam>0$, the \emph{hard-core model on $G$ at fugacity $\lam$} is a probability distribution on $\cI(G)$, where each $I\in\cI(G)$ occurs with probability proportional to $\lam^{|I|}$. Writing $\bI$ for the random independent set, we have
\[
\Pr(\bI=I) = \frac{\lam^{|I|}}{Z_G(\lam)},
\]
where the normalising term in the denominator is the \emph{partition function} (or \emph{independence polynomial}) $Z_G(\lam)=\sum_{I\in\cI(G)} \lam^{|I|}$. 
The \emph{occupancy fraction} is $\EE|\bI|/|V(G)|$. Note that this is a lower bound on the proportion of vertices in a largest independent set of $G$.

There are many equivalent definitions for a fractional colouring of $G$, the most relevant one for our purposes being the following. A \emph{fractional colouring} of $G$ is a weight assignment $w: \cI(G) \to [0,1]$ which satisfies that for every vertex $v\in V(G)$, the sum of the weights of all the independent sets that contain $v$ is at least $1$. The total weight $w(G)$ of $w$ is then the sum of all the weights of the independent sets of $G$. 
The fractional chromatic number $\chi_f(G)$ of $G$ is then the minimum total weight of a fractional colouring of $G$. It is straightforward to show that $|V(G)|/\chi_f(G)$ is a lower bound on the size of a largest independent set in $G$.


We have presented in an earlier work~\cite[Lemma 3]{DJKP20} how to construct a fractional colouring of a graph $G$ using a probability distribution over the independent sets of any induced subgraph $H$ of $G$. Informally speaking, if this probability distribution has the property that for every vertex $v\in H$, either $v$ has a good chance of belonging to the corresponding random independent set $\bI$, or the expected number of neighbours of $v$ in $\bI$ is large, then we obtain an upper bound on $\chi_f(G)$. This is done using a greedy fractional colouring algorithm, and we refer the reader to~\cite{DJKP20} or~\cite[Section 21.3]{MoRe02} for a description of this algorithm.
Note that the statement in~\cite{DJKP20} is significantly more general than what we need here: to obtain the statement below we take $r=1$, rename $\alpha_1,\alpha_2$ to $\alpha,\beta$, and trade the colour-measure definition of fractional colouring for the conceptually simpler weighted independent set version. 
The conclusion of~\cite[Lemma 3]{DJKP20} gives what we want here as when each vertex $v$ is fractionally coloured with an interval of the form $[0, b_v)$ (for some parameters $b_v>0$) the total weight of colour needed is $\max_{v\in V} \{b_v\}$ which implies $\chi_f(G)\le \max_{v\in V}\{b_v\}$.

\begin{lemma}[{\cite{DJKP20}, cf.~also~\cite[Section 21.3]{MoRe02}}]\label{lem:chifalg}
Let $G$ be a graph of maximum degree $\Delta$, and $\alpha,\beta>0$ be positive reals. 
Suppose that for every induced subgraph $H \subseteq G$, there is a probability distribution on $\cI(H)$ such that, writing $\bI_H$ for the random independent set from this distribution, for each $v\in V(H)$ we have
\[
\alpha\Pr(v\in\bI_H) + \beta\EE|N_H(v)\cap \bI_H| \ge 1.
\]
Then $\chi_f(G)\le \alpha +\beta\Delta$.
\end{lemma}

In fact the method we use to lower bound occupancy fractions follows a similar principle: we show that either a vertex is likely to be in the random independent set $\bI$ or the expected number of its neighbours $\bI$ is large. 
These techniques were applied in~\cite{DJKP20} to the special case of triangle-free graphs and here we show that the arguments permit substantially greater generality. 
This is a significant step.  For when lower bounding the independence number one can delete vertices, e.g.\ to remove triangles, and look in the remaining graph. But this does not work immediately in our setting as e.g.\ fractional colouring requires a condition to be satisfied for every vertex. 

Essentially we show that the quantitative properties of the hard-core model that underlie the results of~\cite{DJKP20} degrade ``smoothly'' in the presence of triangles. 
In order to do this we pose and solve a more involved optimisation problem, and this hints at the strength of our approach: it is sufficiently general that carefully chosen values of the parameters yield stronger conclusions than seemingly ``more natural'' choices. 
That is, we state Theorem~\ref{thm:main}\ref{itm:occupancy} in terms of a uniform random independent set, while, on the other hand, na\"ively carrying out our argument with $\lam=1$ (which corresponds to a uniform random independent set) does not yield a conclusion of the desired strength; a smaller value of $\lam$ must be chosen.
That is, Theorem~\ref{thm:occupancy} below gives a lower bound on $\EE|\bI|$ that is not monotone increasing in $\lam$ despite $\EE|\bI|$ itself being monotone increasing in $\lam$.

The function $W:[-1/e,\infty)\to[-1,\infty)$ is the inverse of $z\mapsto ze^z$, also known as the \emph{Lambert $W$ function}. It is monotone increasing and satisfies (see e.g.~\cite{CGH+96})%
\renewcommand{\theequation}{\Roman{equation}}%
\begin{align}
W(x)=\log x - \log\log x + o(1) &\text{ and} \label{Wprop1}\\
W((1+o(1))x)=W(x)+o(1) &\text{ as $x\to\infty$, and} \label{Wprop2}\\
e^{-W(y)}=W(y)/y &\text{ for all $y$}.\label{Wprop3}
\end{align}
\noindent
All logarithms in the present paper are natural.

\renewcommand{\theequation}{\arabic{equation}}
\section{The main result}\label{sec:sparsenbhds}

We next discuss our main result, Theorem~\ref{thm:main}, in slightly deeper context. We in fact show a sharp, general lower bound on occupancy fraction for graphs of bounded local triangle fraction, to which Theorem~\ref{thm:main}\ref{itm:occupancy} is corollary.

\begin{theorem}\label{thm:occupancy}
Suppose $f,\Delta,\lam$ satisfy that $f\le \Delta^2+1$ and, as $f\to\infty$, that
\[
\Delta\log(1+\lam)=\omega(1)
\quad\text{ and }\quad
\frac{2{(\Delta\log(1+\lam))}^2}{fW(\Delta\log(1+\lam))}=o(1).
\]
In any graph $G$ of maximum degree $\Delta$ in which the neighbourhood of every vertex spans at most $\Delta^2/f$ edges, 
writing $\bI$ for the random independent set from the  hard-core model on $G$ at fugacity $\lam$, the occupancy fraction satisfies
\[
\frac{1}{|V(G)|}\EE|\bI| \ge (1+o(1))\frac{\lam}{1+\lam}\frac{W(\Delta\log(1+\lam))}{\Delta\log(1+\lam)}.
\]
\end{theorem}


This may be viewed as generalising~\cite[Thm.~3]{DJPR18}.
By monotonicity of the occupancy fraction in $\lam$ (see e.g.~\cite[Prop.~1]{DJPR18}), and the fact that a uniform choice from $\cI(G)$ is a hard-core distribution with $\lam=1$, Theorem~\ref{thm:main}\ref{itm:occupancy} follows from Theorem~\ref{thm:occupancy} with $\lam = \sqrt{f}/\Delta$.
Theorem~\ref{thm:occupancy} is asymptotically optimal. 
More specifically, in~\cite{DJPR18} it was shown how the analysis of~\cite{BST16} yields that, for any fugacity $\lam=o(1)$ in the range allowed in Theorem~\ref{thm:occupancy}, the random $\Delta$-regular graph (conditioned to be triangle-free) with high probability has occupancy fraction asymptotically equal to the bound in Theorem~\ref{thm:occupancy}.
In Section~\ref{sec:sharpness}, we show our methods break down for $\lam$ outside this range, so that new ideas are needed for any improvement in the bound for larger $\lam$.

Moreover, the asymptotic bounds of Theorems~\ref{thm:trianglecount} and~\ref{thm:main} cannot be improved, for any valid choice of $f$ as a function of $\Delta$, by more than a factor of between $2$ and $4$. This limits the hypothetical range of $\lam$ in Theorem~\ref{thm:occupancy}. This follows by considering largest independent sets in a random regular construction or in a suitable blow-up of that construction~\cite{She83}; see Section~\ref{sec:sharpness}.

Observe that Theorem~\ref{thm:main}\ref{itm:fractional} trivially fails with a global, rather than local, triangle fraction condition by the presence of a $(\Delta+1)$-vertex clique as a subgraph.
So Theorems~\ref{thm:trianglecount} and~\ref{thm:main} may appear incompatible, since the former has a global condition, while the latter has a local one. Nevertheless, either assertion in Theorem~\ref{thm:main} is indeed (strictly) stronger.

\begin{proof}[Proof of Theorem~\ref{thm:trianglecount}]
Without loss of generality we may assume that $\eps>0$ is small enough so that $(1/2-\eps^2)(1-3\eps^2)(1-\eps^2) \ge 1/2-\eps$.
Let $G$ be a graph on $n$ vertices of maximum degree $\Delta$ with at most $\Delta^2n/f$ triangles.
Call $v\in V(G)$ \emph{bad} if the number of triangles of $G$ that contain $v$ is greater than $\eps^{-2} \Delta^2/f$. Let $B$ be the set of all bad vertices. Note that $3\Delta^2n/f > |B| \eps^{-2}\Delta^2/f$ and so $|B| < 3\eps^2n$.
Let $H$ be the subgraph of $G$ induced by the subset $V(G)\setminus B$. Then $H$ is a graph of maximum degree $\Delta$ on at least $(1-3\eps^2)n$ vertices such that the neighbourhood of any vertex spans at most $\Delta^2/(\eps^2 f)$ edges. Provided we take $f$ large enough, either of~\ref{itm:occupancy} and~\ref{itm:fractional} in Theorem~\ref{thm:main} (with parameter $\eps$ in Theorem~\ref{thm:main} being $\eps^2$ here) implies that $H$, and thus $G$, contains an independent set of size
\begin{align}
(1/2-\eps^2)\frac{(1-3\eps^2)n}{\Delta}\log(\eps^2 f) 
&\ge (1/2-\eps^2)(1-3\eps^2)(1-\eps^2)\frac{n}{\Delta}\log f \\
&\ge (1/2-\eps)\frac{n}{\Delta}\log f,
\end{align}
where on the first line we used that $\eps^2 f \ge f^{1-\eps^2}$ for $f$ large enough.
\end{proof}

\section{An analysis of the hard-core model}\label{sec:hardcore}

A crucial ingredient in the proofs is an occupancy guarantee from the hard-core model, which we establish in Lemma~\ref{lem:hcmbound} below.
This refines an analysis given in~\cite{DJPR18}.
Given $G$, $I\in\cI(G)$, and $v\in V(G)$, let us call a neighbour $u\in N(v)$ of $v$ \emph{externally uncovered by $I$} if $u\notin N(I\setminus N(v))$.

\begin{lemma}\label{lem:genhcm}
Let $G$ be a graph and $\lam>0$. Let $\bI$ be an independent set drawn from the hard-core model at fugacity $\lam$ on $G$.
\begin{enumerate}
\item\label{itm:pv}
For every $v\in V(G)$, writing $\bF_v$ for the subgraph of $G$ induced by the neighbours of $v$ externally uncovered by $\bI$, 
\[
\Pr(v\in\bI)\ge \frac{\lam}{1+\lam}{(1+\lam)}^{-\EE|V(\bF_v)|}.
\]
\item\label{itm:genhcm}
Moreover,
\[
\EE|\bI| \ge \frac{\lam}{1+\lam}|V(G)|{(1+\lam)}^{-\frac{2|E(G)|}{|V(G)|}}.
\]
\end{enumerate}
\end{lemma}
\begin{proof}
The first part follows from two applications of the spatial Markov property of the hard-core model. 
This property is a general statement about the model which states that for $X\subset V(G)$, conditioned on some event of the form $\{\bI\cap (V(G)\setminus X)=J\}$, $\bI\cap X$ is distributed according to the hard-core model on $G[X\setminus N_G(J)]$.
This can be proved using the definition of the model: conditioned on $\{\bI\cap (V(G)\setminus X)=J\}$, the possible values for $\bI\cap X$ are precisely the independent sets in $G[X\setminus N_G(J)]$, and it can be verified from first principles that each occurs with the correct probability (see also e.g.~\cite{DJPR17}).
First, we have 
\[
\Pr(v\in\bI) = \frac{\lam}{1+\lam}\Pr(\bI\cap N(v) = \emptyset),
\]
because conditioned on a value $\bI\setminus\{v\}=J$ such that $J\cap N(v) = \emptyset$ there are two realisations of $\bI$, namely $J$ and $J\cup\{v\}$, giving
\[
\Pr(v\in\bI\mid J)=\frac{\lam^{|J|+1}}{\lam^{|J|}+\lam^{|J|+1}}=\frac{\lam}{1+\lam},
\]
and conditioned on $\bI\setminus\{v\}=J$ such that $J\cap N(v) \ne\emptyset$, $v$ cannot be in $\bI$.  
This is the spatial Markov property with $X=\{v\}$.

Second, the spatial Markov property with $X= N(v)$ gives that $\bI\cap N(v)$ is a random independent set drawn from the hard-core model on $\bF_v$. 
Then $\bI\cap N(v)=\emptyset$ if and only if this random independent set in $\bF_v$ is empty. It follows that
\[
\Pr(\bI\cap N(v) = \emptyset) = \EE \frac{1}{Z_{\bF_v}(\lam)} \ge \EE {(1+\lam)}^{-|V(\bF_v)|} \ge {(1+\lam)}^{-\EE|V(\bF_v)|},
\]
since the graph on $|V(\bF_v)|$ vertices with largest partition function is the graph with no edges, and by convexity. This completes the proof of~\ref{itm:pv}.

By the fact that $|V(\bF_v)|\le\deg(v)$ we also have for all $v\in V(G)$ that
\[
\Pr(v\in\bI) \ge \frac{\lam}{1+\lam}{(1+\lam)}^{-\deg(v)}.
\]
Then~\ref{itm:genhcm} follows by convexity:
\begin{align}
\EE|\bI| = \sum_{v\in V(G)}\Pr(v\in\bI) 
  &\ge \frac{\lam}{1+\lam}|V(G)|\sum_{v\in V(G)}{(1+\lam)}^{-\deg(v)} \\&\ge \frac{\lam}{1+\lam}|V(G)|{(1+\lam)}^{-\frac{2|E(G)|}{|V(G)|}}.\qedhere
\end{align}

\end{proof}

\begin{lemma}\label{lem:hcmbound}
Let $G$ be a graph of maximum degree $\Delta$ in which the neighbourhood of every vertex in $G$ spans at most $\Delta^2/f$ edges and $\lambda,\alpha,\beta>0$. 
Let $\bI$ be an independent set drawn from the hard-core model at fugacity $\lam$ on $G$.
Then we have, for every $v\in V(G)$,
\begin{align}\label{eqn:hcmboundlocal}
\alpha \Pr(v\in\bI) + &\beta\EE|N(v)\cap \bI| \\
&\ge 
\frac{\lam}{1+\lam}\min_{z\ge0}\left(\alpha{(1+\lam)}^{-z} + \beta z{(1+\lam)}^{-\frac{2\Delta^2}{fz}}\right).
\end{align}
Moreover,
\begin{align}
\label{eqn:hcmbound}
\frac{1}{|V(G)|}\EE|\bI| & \ge \frac{\lam}{1+\lam} \min_{z\ge 0}\max\left\{{(1+\lam)}^{-z}, \frac{z}{\Delta}{(1+\lam)}^{-\frac{2\Delta^2}{fz}}\right\}.
\end{align}
\end{lemma}

\begin{proof}
Write $\bF_v$ for the graph induced by the neighbours of $v$ externally uncovered by $\bI$ and $z_v=\EE|V(\bF_v)|$. 
By Lemma~\ref{lem:genhcm}\ref{itm:pv} we have
\[
\Pr(v\in\bI) \ge \frac{\lam}{1+\lam}{(1+\lam)}^{-z_v}.
\]
For the other term, we apply Lemma~\ref{lem:genhcm}\ref{itm:genhcm} to the graph $\bF_v$, for which by assumption $\frac{2|E(\bF_v)|}{|V(\bF_v)|} \le \frac{2\Delta^2}{f|V(\bF_v)|}$. If $\bJ$ is an independent set drawn from the hard-core model at fugacity $\lam$ on $\bF_v$,
then by convexity
\begin{align}
\EE|N(v)\cap\bI| = \EE|\bJ|
& \ge \frac{\lam}{1+\lam}\EE\left[|V(\bF_v)|{(1+\lam)}^{-\frac{2\Delta^2}{f|V(\bF_v)|}}\right]\\
& \ge \frac{\lam}{1+\lam}z_v{(1+\lam)}^{-\frac{2\Delta^2}{fz_v}},
\end{align}
so~\eqref{eqn:hcmboundlocal} follows.
For~\eqref{eqn:hcmbound}, by above we may bound $\EE|\bI|$ in two distinct ways:
\begin{align}
\EE|\bI|
& \ge |V(G)| \frac{\lam}{1+\lam}{(1+\lam)}^{-z} \quad \text{and}\\
\EE|\bI| 
& \ge \frac{1}{\Delta}\sum_{v\in V(G)}\sum_{u\in N(v)}\Pr(u\in \bI) 
 \ge \frac{1}{\Delta}\sum_{v\in V(G)}\EE|N(v)\cap\bI| \\
& \ge \frac{|V(G)|}{\Delta}\frac{\lam}{1+\lam}z{(1+\lam)}^{-\frac{2\Delta^2}{fz}}
\end{align}
where $z=\frac{1}{|V(G)|}\sum_{v\in V(G)} z_v$ is the expected number of externally uncovered neighbours of a uniformly random vertex and for the final inequality we use the convexity of $x\mapsto x e^{-k/x}$ on $[0,\infty)$ for any $k \in \mathbb{R}$. Now~\eqref{eqn:hcmbound} follows.
\end{proof}

\section{The proofs}\label{sec:proofs}

\begin{proof}[Proof of Theorem~\ref{thm:occupancy}]
In order to obtain the desired result we evaluate the right hand side of~\eqref{eqn:hcmbound} in Lemma~\ref{lem:hcmbound}.
For fixed positive parameters $\lambda,\Delta,f$,
the first argument of the maximisation in~\eqref{eqn:hcmbound} is decreasing from $1$ to $0$ as $z$ grows from $0$ to $+\infty$, while the second is increasing from $0$ to $+\infty$. Therefore, there exists a unique $z_0>0$ satisfying
\begin{align}\label{eqn:z}
{(1+\lam)}^{-z_0} = \frac{z_0}{\Delta} {(1+\lam)}^{-\frac{2\Delta^2}{fz_0}},
\end{align}
and the minimum in $\eqref{eqn:hcmbound}$ is attained at $z=z_0$. 
We now seek a suitable upper bound on $z_0$, and change variables in order to simplify the above equation.
Writing $\Lambda=\Delta\log(1+\lam)$ and $y=z_0\log(1+\lam)$, we can rewrite~\eqref{eqn:z} as
\[
ye^y = \Lambda e^{\frac{2\Lambda^2}{fy}} \ge \Lambda,
\]
where the lower bound comes from the fact that all fixed parameters are positive.
Using the fact that $W$ is monotone increasing we obtain $W(\Lambda) \le y$, which we reinject into the equation in order to obtain
\begin{align*}
ye^y &= \Lambda e^{\frac{2\Lambda^2}{fy}} \le \Lambda e^{\frac{2\Lambda^2}{fW(\Lambda)}},\text{ and hence}\\
y &\le W\Big(\Lambda e^{\frac{2\Lambda^2}{fy}}\Big).
\end{align*}

In terms of the original parameters, this upper bound becomes
\begin{align}
z_0\log(1+\lam) \le W\Big(\Delta\log(1+\lam)e^{\frac{2{(\Delta\log(1+\lam))}^2}{fW(\Delta\log(1+\lam))}}\Big).
\end{align}
By the assumptions of Theorem~\ref{thm:occupancy}, and using property~\eqref{Wprop2} of $W$, we have that, as $f\to \infty$, 
\begin{align}\label{eqn:Lambda}
\hspace{16pt} z_0\log(1+\lam)\le W(\Delta\log(1+\lam)(1+o(1))) = W(\Delta\log(1+\lam))+o(1).
\end{align}
Substituting this into the first argument in~\eqref{eqn:hcmbound}, and using property~\eqref{Wprop3} of $W$, we obtain as $f\to\infty$
\begin{align}
\frac{1}{|V(G)|}\EE|\bI|
& \ge \frac{\lam}{1+\lam}e^{-W(\Delta\log(1+\lam))+o(1)} 
= (1+o(1))\frac{\lam W(\Delta\log(1+\lam))}{(1+\lam)\Delta\log(1+\lam)}.\qedhere
\end{align}
\end{proof}

\begin{proof}[Proof of Theorem~\ref{thm:main}\ref{itm:fractional}]
Supposing that we have chosen positive values of $\alpha$, $\beta$, and $\lam$, write
\[
g(x) = \frac{\lam}{1+\lam}\left(\alpha{(1+\lam)}^{-x} + \beta x {(1+\lam)}^{-\frac{2\Delta^2}{fx}}\right).
\]
By~\eqref{eqn:hcmboundlocal} in Lemma~\ref{lem:hcmbound} and Lemma~\ref{lem:chifalg}, we have $\chi_f(G)\le \alpha+\beta\Delta$, provided $g(x)\ge 1$ for all $x\ge0$.
It is easy to verify that with $\alpha,\beta,\lam>0$ the function $g$ is strictly convex, so the minimum of $g(x)$ occurs when $g'(x)=0$, or 
\[
{(1+\lam)}^{-x} = \frac{\beta}{\alpha}\left(\frac{1}{\log(1+\lam)} + \frac{2\Delta^2}{fx}\right){(1+\lam)}^{-\frac{2\Delta^2}{fx}}.
\]
As before, let $z_0$ be the unique solution over $\RR^+$ of~\eqref{eqn:z}, and hence~\eqref{eqn:Lambda} is again satisfied. Then by choosing
\begin{align}\label{eq:alphabeta}
\frac{\alpha}{\beta} = \frac{\Delta}{z_0}\left(\frac{1}{\log(1+\lam)} + \frac{2\Delta^2}{fz_0}\right),
\end{align}
the minimum of $g$ occurs at $z_0$. Observe that, if~\eqref{eq:alphabeta} is satisfied, then $g(z_0) \to 0$ as $\alpha \to 0$, and $g(z_0) \to +\infty$ as $\alpha \to +\infty$, therefore we can express values of $\alpha,\beta > 0$ in terms of  $\lam$, $\Delta$, and $f$ such that the equations $g(z_0)=1$ and~\eqref{eq:alphabeta} are satisfied.
Using~\eqref{eqn:z}, this means
\[
g(z_0) = \frac{\lam}{1+\lam} {(1+\lam)}^{-z_0}(\alpha+\beta\Delta) = 1,
\]
and hence by Lemma~\ref{lem:chifalg} we obtain
\[
\chi_f(G) \le \alpha+\beta\Delta = \frac{1+\lam}{\lam} {(1+\lam)}^{z_0}.
\]
We let $\lam$ be such that $\Delta\log(1+\lam)=\sqrt{f}=\omega(1)$ as $f\to\infty$. In this case, by~\eqref{eqn:Lambda} we know that $z_0\log(1+\lam) \le W(\Delta \log(1+\lam))+o(1)$.
Thus
\begin{align*}
\chi_f(G) \le \frac{1+\lam}{\lam} {(1+\lam)}^{z_0} &\le \frac{1+\lam}{\lam} e^{W(\Delta \log(1+\lambda))+o(1)} & \text{by~\eqref{eqn:Lambda}}\\
&\le (1+o(1)) \frac{\Delta}{\sqrt{f}}e^{W\big(\sqrt{f}\big)} & \text{by definition of $\lambda$}\\
&\le (1+o(1)) \frac{\Delta}{W(\sqrt{f})} & \text{by~\eqref{Wprop3}}\\
&\le (1+o(1)) \frac{\Delta}{\log \sqrt{f}} &\text{by~\eqref{Wprop1} }\\
&< (2+\eps)\frac{\Delta}{\log f},
\end{align*}
provided that $f_0$ is taken large enough.
\end{proof}

\section{Sharpness}\label{sec:sharpness}

\subsection{Occupancy fraction}

Since the occupancy fraction is increasing in $\lam$, it might be intuitive that the lower bound on occupancy fraction that results from the proof of Theorem~\ref{thm:occupancy} is also increasing.
This is true only up to a point, just as in~\cite{DJPR18}. 
Already for $\lam$ slightly larger than admissible for Theorem~\ref{thm:occupancy}, under mild assumptions, the method breaks down in the sense that the resulting lower bound is asymptotically smaller. 
In this case a novel analysis would be necessary; there is almost no slack in our treatment of~\eqref{eqn:z}. 

Assume as $f\to\infty$ that $\lambda = o(1)$, $\Delta =f^{O(1)}$,
\[
\lam = \omega\left(\frac{\sqrt{f{(\log f)}^3}}{\Delta}\right) \quad\text{ and }\quad \lam = o\left(\frac{f\log f}{\Delta}\right).
\]
The choice of $\lam$ used to obtain Theorem~\ref{thm:main} is just shy of the above range. 
Substituting the extremes of the interval~\eqref{eqn:Lambda} into the second argument of~\eqref{eqn:hcmbound}, we derive the following two expressions as $f\to\infty$, the larger of which necessarily bounds the best guarantee to expect from our approach.
First,
\begin{align}
\frac{\lam}{1+\lam}&\frac{z}{\Delta}\left.{(1+\lam)}^{-\frac{2\Delta^2}{fz}}\right|_{z\log(1+\lam) = W(\Delta\log(1+\lam))} \\
& =(1+o(1))\frac{\log(\Delta\lam)}{\Delta}\exp\left(-(1+o(1))\frac{2\Delta^2\lam^2}{f\log(\Delta\lam)}\right)\\
& \le \Delta^{-1}\exp\left(-\omega\big({(\log f)}^2\big)\right)= o\left(\frac{\log(\Delta\lam)}{\Delta}\right).
\shortintertext{Second, using the properties of $W$ and the assumed bounds on $\lam$,}
\frac{\lam}{1+\lam}&\frac{z}{\Delta}\left.{(1+\lam)}^{-\frac{2\Delta^2}{fz}}\right|_{z\log(1+\lam) = W\Big(e^{\frac{2{(\Delta\log(1+\lam))}^2}{fW(\Delta\log(1+\lam))}+\log(\Delta\log(1+\lam))}\Big)} \\
& = (1+o(1))\frac{2\Delta\lam^2}{f\log(\Delta\lam)}\exp\left(-\frac{2{(\Delta\log(1+\lam))}^2}{fW\Big(e^{\frac{2{(\Delta\log(1+\lam))}^2}{fW(\Delta\log(1+\lam))}+\log(\Delta\log(1+\lam))}\Big)}\right) \\
& = (1+o(1))\frac{2\Delta\lam^2}{f\log(\Delta\lam)}\exp\left(-\frac{W(\Delta\log(1+\lam))}{1-(1+o(1))\frac{f\log(\Delta\lam/f)\log(\Delta\lam)}{2{(\Delta\lam)}^2}}\right) \\
& = (1+o(1))\frac{2\Delta\lam^2}{f\log(\Delta\lam)}{\left(\frac{\log(\Delta\lam)}{\Delta\lam}\right)}^{\frac{1}{1-(1+o(1))\frac{f\log(\Delta\lam/f)\log(\Delta\lam)}{2{(\Delta\lam)}^2}}} \\
& = (1+o(1))\frac{2\lam}{f} = o\left(\frac{\log(\Delta\lam)}{\Delta}\right).
\end{align}

\subsection{Independence number and fractional chromatic number}

By an analysis of the random regular graph~\cite{FrLu92}, there are triangle-free $\Delta$-regular graphs $G_\Delta$ in which as $\Delta\to\infty$ the largest independent set has size at most 
\[
(2+o(1))\frac{|V(G_\Delta)|}{\Delta}\log \Delta.
\]
So $G_\Delta$ certifies Theorems~\ref{thm:trianglecount} and~\ref{thm:main} to be sharp up to an asymptotic factor $2$ provided $f=\Delta^{2-o(1)}$ as $\Delta\to\infty$. For smaller $f$, let us for completeness reiterate an observation from~\cite{She83}. For $f<\Delta/2$, let $d=f-1$ and let $bG_d$ be the graph obtained from $G_d$ by substituting each vertex with a clique of size $b=\lfloor\Delta/f\rfloor$. Then $bG_d$ is regular of degree $bf-1\le\Delta$ such that each neighbourhood contains at most $b^2f/2\le \Delta^2/(2f)$ edges, and so  $bG_d$ has at most $\Delta^2|V(bG_d)|/f$ triangles. In $G_d$ the largest independent set is of size
\begin{align}
(2+o(1))\frac{|V(G_d)|}{d}\log d = (2+o(1))\frac{|V(bG_d)|}{\Delta}\log f.
\end{align}
The same is clearly true in $bG_d$, and this is an asymptotic factor $4$ greater than the lower bound in Theorem~\ref{thm:trianglecount}. Last, observe that for $f\ge \Delta/2$ and $f\le\Delta^{2-\Omega(1)}$, $G_\Delta$ certifies that Theorems~\ref{thm:trianglecount} and~\ref{thm:main} are at most an asymptotic factor $4$ from extremal, and so this holds throughout the range of $f$.

\section{A more basic setting}\label{sec:basic}

Due to the close link with off-diagonal Ramsey numbers, we wonder, what occurs when we drop the degree bounding parameter?
This may be variously interpreted. For example, over graphs on $n$ vertices with at most $n^3/f$ triangles, what is the best lower bound on the independence number?

One can deduce nearly the correct answer to an alternative, local version of this question: over graphs on $n$ vertices such that each vertex $v$ is contained in at most $\deg(v)^2/f$ triangles, where $2\le f\le (n-1)^2+1$, what is the best lower bound on the independence number?
By a comparison of Theorem~\ref{thm:trianglecount} and Tur\'{a}n's theorem (applied to a largest neighbourhood), as $f\to\infty$ there must be an independent of size at least
\begin{align}\label{eqn:basic}
(1+o(1))\frac{f}{1+\sqrt{1+\frac{2f^2}{n\log \min\{f,n\}}}}.
\end{align}
This expression is asymptotic to $f/2$ if $f = o(\sqrt{n\log n})$ and asymptotic to $\sqrt{0.5n\log \min\{f,n\}}$ if $f = \omega(\sqrt{n\log n})$, so this in particular extends Shearer's bound on off-diagonal Ramsey numbers to cover any $f\ge n^{1-o(1)}$.
Over the range of $f$ as a function of $n$,~\eqref{eqn:basic} is asymptotically sharp up to some reasonably small constant factor by considering the final output of the triangle-free process~\cite{BoKe13+,FGM13+} or a blow-up of that graph by cliques.

We remark that the same argument as above yields a similar bound as in~\eqref{eqn:basic} for a more local version, where $f=\deg(v)^a+1$ for some fixed $a\in[0,2]$.

By repeatedly extracting independent sets of the size guaranteed in~\eqref{eqn:basic} (cf.~\cite[Lemma~4.1]{CJKP18+}), 
it follows that as $f\to\infty$ the chromatic number is at most
\begin{align}
(2+o(1))\left(1+\sqrt{1+\frac{2f^2}{n\log \min\{f,n\}}}\right)\frac{n}{f}.
\end{align}
Is this the correct asymptotic order for the largest list chromatic number?
Is the extra factor $2$ unnecessary for the chromatic number? Even improving the bound only for the fractional chromatic number by a factor $2$ would be very interesting.
This generalises recent conjectures of Cames van Batenburg and a subset of the authors~\cite{CJKP18+}, for the triangle-free case $f=(n-1)^2+1$.

\subsection*{Acknowledgements}
We are grateful to Matthew Jenssen, Will Perkins and Barnaby Roberts for helpful discussions, particularly in relation to the analysis in Section~\ref{sec:hardcore}.
We are thankful to the anonymous referees for their thorough reading and helpful comments.

\subsection*{Added remarks}

Recalling the definition of fractional colouring given in Subsection~\ref{sub:prelim}, it is sometimes useful to give a concrete meaning to the weights in terms of colours. In this case, the set of colours is the interval $[0,w(G)]$, and each independent set $I\in \cI(G)$ is assigned a disjoint subinterval of length $w(I)$. This is necessary in order to consider \emph{local} fractional colourings, where each vertex has only a subinterval of allowed colours whose length depends on some ``local'' parameters, such as its degree. A version of the present work expressed in this extra generality can be found in the doctoral thesis of the last author~\cite{Pir19}. 

Moreover, following this work, Conjecture~\ref{conj:chi} has been confirmed in a stronger form within a general framework (which also includes local colouring among other things) due to Sereni and three of the authors~\cite{DKPS20+}.
\bibliography{sparsenbhds}
\bibliographystyle{habbrv}

\end{document}